\documentclass{elsarticle}

\usepackage{amsmath,amsthm,amsfonts,amssymb}
\usepackage{fullpage,enumerate}
\usepackage{multirow}
\usepackage{color}
\usepackage{bm}
\usepackage{caption,subcaption,graphicx}
\usepackage[utf8]{inputenc}

\newtheorem{thm}{Theorem}
\newtheorem{lem}{Lemma}
\newtheorem{res}{Result}
\newtheorem{prop}{Proposition}
\newtheorem{cor}{Corollary}

\theoremstyle{definition}
\newtheorem{defn}{Definition}
\newtheorem{rem}{Remark}
\newtheorem{exmp}{Example}

\newcommand{\ZZ}{\mathbb{Z}}
\newcommand{\CC}{\mathbb{C}}
\newcommand{\RR}{\mathbb{R}}
\newcommand{\TT}{\mathbb{T}}
\newcommand{\dmat}{\mathcal{M}}
\newcommand{\grp}{\mathcal{G}}

\newcommand{\tor}{\partial\mathbb{D}}
\newcommand{\ddet}{\mathcal{Q}}

\newcommand{\z}{\mathbf{z}}
\newcommand{\K}{\CC(\z)}
\newcommand{\Kr}{\RR(\z)}

\newcommand{\Laur}{\CC[\z^{\pm1}]}

\begin{document}
\begin{frontmatter}
\title{Multivariate Tight Wavelet Frames with Few Generators and High Vanishing Moments} 
\author[yon]{Youngmi~Hur\fnref{fn1}}
\ead{yhur@yonsei.ac.kr}
\author[jhu]{Zachary Lubberts\fnref{fn2}}
\ead{zlubber1@jhu.edu}
\author[cp]{Kasso~A.~Okoudjou\fnref{fn3}}
\ead{kasso@math.umd.edu}
\fntext[fn1]{This research was supported in part by the National Research Foundation of Korea (NRF) [Grant Number 20151009350].}%
\fntext[fn2]{Corresponding Author.}
\fntext[fn3]{This research was partially supported by a grant from the Simons Foundation $\# 319197$, the U. S. Army Research Office grant W911NF1610008, the U. S. National Foundation (NSF) grant 1814253, and an MLK  visiting professorship at MIT.}%

\address[yon]{Department of Mathematics, Yonsei University, Seoul 03722, Korea.}%
\address[jhu]{Department of Applied Mathematics and Statistics, 100 Whitehead Hall, Johns Hopkins University, 3400 N Charles St, Baltimore, MD 21218, USA.}%
\address[cp]{Department of Mathematics, University of Maryland College Park, College Park, MD 20742, USA.}%

%%%%%%ABSTRACT%%%%%%
\begin{abstract}
Tight wavelet frames are computationally and theoretically attractive, but most existing multivariate constructions have various drawbacks, including low vanishing moments for the wavelets, or a large number of wavelet masks. We further develop existing work combining sums of squares representations with tight wavelet frame construction, and present a new and general method for constructing such frames. Focusing on the case of box splines, we also demonstrate how the flexibility of our approach can lead to tight wavelet frames with high numbers of vanishing moments for all of the wavelet masks, while still having few highpass masks: in fact, we match the best known upper bound on the number of highpass masks for general box spline tight wavelet frame constructions, while typically achieving much better vanishing moments for all of the wavelet masks, proving a nontrivial lower bound on this quantity.
\end{abstract}

%%%%%%KEYWORDS%%%%%%
\end{frontmatter}

\section{Introduction}
\label{S:intro}

Recent work in tight wavelet frame construction \cite{AlgPersI, LaiStock} has shown how sos representations for certain nonnegative trigonometric polynomials may be used to create highpass masks generating a tight wavelet frame for $L^{2}(\RR^{n})$. This construction makes use of the unitary extension principle (UEP) conditions on a collection of trigonometric polynomials, which are sufficient for the highpass masks to generate a tight wavelet frame \cite{DHRS,HanOrtho}. In this setting, we call trigonometric polynomials masks, which are lowpass when equal to one at $\omega=0$, and are highpass or wavelet when they are equal to zero there. Considering the case of dyadic dilation for now, when we are given a lowpass mask $\tau$ and a collection of highpass masks $\{q_{\ell}\}_{\ell=1}^{r}$, we say that these satisfy the UEP conditions when $$\tau(\omega)\overline{\tau(\omega+\gamma)}+\sum_{\ell=1}^{r}q_{\ell}(\omega)\overline{q_{\ell}(\omega+\gamma)}=\begin{cases}1&\text{ if }\gamma=0,\text{ and}\\0&\text{ if }\gamma\in\{0,\pi\}^{n}\setminus\{0\},\end{cases}$$ for all $\omega\in\TT^{n}:=[-\pi,\pi]^{n}$. These conditions necessitate that $f(\tau;\omega)=1-\sum_{\gamma\in\{0,\pi\}^{n}}|\tau(\omega+\gamma)|^{2}\geq0$ for all $\omega\in\TT^{n}$, which is called the sub-QMF condition, and when equality holds for all $\omega\in\TT^{n},$ then $\tau$ is said to satisfy the QMF condition. 

In fact, the UEP conditions imply that $f(\tau;\cdot)$ is a sum of squares, and the work in \cite{AlgPersI,LaiStock} shows the converse. When $f(\tau;\cdot)$ has a sum of hermitian squares representation $\sum_{j=1}^{J}|g_{j}(2\omega)|^{2}$ with trigonometric polynomials $g_{j},1\leq j\leq J$, they construct highpass masks satisfying the UEP conditions with $\tau$. In the special case of dyadic dilation, their construction proceeds as follows: Construct the column vectors $H(\omega)=[\tau(\omega+\gamma)]_{\gamma\in\{0,\pi\}^{n}}$ and $G(\omega)=[g_{j}(\omega)]_{j=1}^{J}$, and the Fourier transform matrix $X(\omega)=2^{-n/2}[e^{i(\omega+\gamma)\cdot\nu}]_{\gamma\in\{0,\pi\}^{n},\nu\in\{0,1\}^{n}}$. Then using block matrix notation, we have $$[H(\omega)\;H(\omega)G(2\omega)^{*}\;(I-H(\omega)H(\omega)^{*})X(\omega)]\left[\begin{array}{c}H(\omega)^{*}\\ G(2\omega)H(\omega)^{*}\\ X(\omega)^{*}(I-H(\omega)H(\omega)^{*})\end{array}\right]=I,$$ which are just another way of writing the UEP conditions, with the highpass masks $q_{1,j}(\omega)=\overline{g_{j}(2\omega)}\tau(\omega)$, $1\leq j\leq J$, $q_{2,\nu}(\omega)=2^{-n/2}(e^{i\omega\cdot\nu}-\tau(\omega)\sum_{\gamma\in\{0,\pi\}^{n}}\overline{\tau(\omega+\gamma)}e^{i(\omega+\gamma)\cdot\nu}),$ $\nu\in\{0,1\}^{n}$, which we can read off from the first row of the left-hand matrix. This means that the wavelet system generated by $(\tau,\{q_{1,j}\},\{q_{2,\nu}\})$ is a tight wavelet frame.

Rewriting this matrix product and using the relationship $G(2\omega)^{*}G(2\omega)=1-H(\omega)^{*}H(\omega)$, we see that this could be written as $$[H(\omega)\;(I-H(\omega)H(\omega)^{*})X(\omega)]\left[\begin{array}{cc}1+(1-H(\omega)^{*}H(\omega))&0\\ 0&I\end{array}\right]\left[\begin{array}{c}H(\omega)^{*}\\ X(\omega)^{*}(I-H(\omega)H(\omega)^{*})\end{array}\right]=I.$$ In \cite{SLP}, this is interpreted as a scaling of the Laplacian pyramid matrix $[H(\omega)\;(I-H(\omega)H(\omega)^{*})X(\omega)],$ and it was shown that this scaling matrix is the unique diagonal matrix which makes this product equal to the identity. There, the scaling matrix was factored under the assumption that $2-H(\omega)^{*}H(\omega)=|g(2\omega)|^{2}$ for some trigonometric polynomial $g$, giving rise to a modified lowpass mask $g(2\cdot)\tau$, but the construction in \cite{AlgPersI} instead assumes that $2-H(\omega)^{*}H(\omega)$ factorizes as $1+\sum_{j=1}^{J}|g_{j}(2\omega)|^{2}$. Combining these ideas, if there is a sum of squares representation for $2-H(\omega)^{*}H(\omega)$ as $\sum_{j=0}^{J}|g_{j}(2\omega)|^{2}$ where $g_{0}(0)=1$, then modifying the lowpass mask to be $g_{0}(2\cdot)\tau(\cdot)$ and constructing the highpass masks $q_{1,j}$ as above leads to a tight wavelet frame. Moreover, these constructions which modify the lowpass mask do not require the original lowpass mask to satisfy the sub-QMF condition, but after modifying, the new lowpass mask will satisfy this condition. We give more details about such constructions in a more general context in Section~\ref{S:slp}.

One downside to both of these constructions is that they rely on the UEP, which may result in highpass masks having suboptimal vanishing moments. A highpass mask's number of vanishing moments is the order of its root at $\omega=0$, and is related to approximation rates for the corresponding wavelet system \cite{DHRS}. A method for correcting this introduces a vanishing moment recovery (vmr) function, and uses highpass masks $\{q_{\ell}\}_{\ell=1}^{r}$ satisfying the oblique extension principle (OEP) conditions with the lowpass mask $\tau$ and vmr function $S$ to generate a tight wavelet frame \cite{DHRS}: $$S(2\omega)\tau(\omega)\overline{\tau(\omega+\gamma)}+\sum_{\ell=1}^{r}q_{\ell}(\omega)\overline{q_{\ell}(\omega+\gamma)}=\begin{cases}S(\omega)&\text{ if }\gamma=0,\text{ and }\\0&\text{ if }\gamma\in\{0,\pi\}^{n}\setminus\{0\},\end{cases}$$ where we continue to assume the setting of dyadic dilation for the time being. In \cite{LaiStock}, a strong condition on the vmr function and lowpass mask was found which results in tight wavelet frames with maximum vanishing moments, and may be viewed as an ``oblique QMF condition'' (though this terminology is not used there). In this paper, we consider a weaker version of this condition, which we call the ``oblique sub-QMF condition'' on the vmr function and lowpass mask: $$f(S,\tau;\omega)=\frac{1}{S(2\omega)}-\sum_{\gamma\in\{0,\pi\}^{n}}\frac{|\tau(\omega+\gamma)|^{2}}{S(\omega+\gamma)}\geq0\text{ for all }\omega\in\TT^{n}.$$ As before, when equality holds for all $\omega\in\TT^{n}$, this is the oblique QMF condition. This analogy extends even further, however: 
\begin{quote}
In Theorem~\ref{th:osqmf}, for a lowpass mask $\tau$ and rational trigonometric polynomial vmr function $S$, we establish the equivalence between the oblique sub-QMF condition and the existence of rational trigonometric polynomial highpass masks satisfying the OEP conditions with this $S$ and $\tau$. 
\end{quote}
Moreover, this structure allows us to show that the given masks generate a tight wavelet frame only using some mild conditions on the vanishing moment recovery function $S$:
\begin{quote}
In Theorem~\ref{th:osqmftwf}, we show that under the additional assumptions that $S$ is continuous at $0$ with $S(0)=1$, and that $S$ and $1/S$ belong to $L^{\infty}(\TT^{n})$, the wavelet system generated by the rational trigonometric polynomial highpass masks in Theorem~\ref{th:osqmf} is a tight wavelet frame.
\end{quote}

There are two primary motivations for considering the oblique sub-QMF condition rather than the oblique QMF condition: First, while it is possible to use the Fourier transform of the autocorrelation function for $\phi$ to obtain a vanishing moment recovery function which satisfies the oblique QMF condition with $\tau$ \cite{LaiStock}, this may be quite difficult to compute, especially as the dimension increases; Second, even if we are able to find this $S$, this choice may lead to a very high number of highpass masks in the constructed tight wavelet frame, since we have $2^{n}\times K$ highpass masks, where $n$ is the spatial dimension and $K$ is the number of sos generators for $1/S$. This motivates us to consider changing $S$ to have a simple form, which still gives us high vanishing moments for the constructed highpass masks, but may be written as a sum of few squares. We revisit this in Section~\ref{S:boxspl}, where we show that for box splines, it is possible to construct $S$ with $K=1$ that gives a collection of highpass masks all having at least $m$ vanishing moments, where $m$ is at least the accuracy number of any separable factor of the lowpass mask. This can lead to a dramatic reduction in the number of highpass masks for the constructed wavelet system, while only giving up a small fraction of the vanishing moments: see Examples~\ref{ex:manymasks} and \ref{ex:fewmasks}.

Moreover, if we allow rational trigonometric polynomial highpass masks, we are able to eliminate the assumption about the existence of a sum of squares representation for $f(S,\tau;\cdot)$, because we have observed that nonnegative rational trigonometric polynomials have representations as sums of squares of rational trigonometric polynomials in Corollary~\ref{c:trigsors}. This stands in contrast to the UEP constructions considered above, since in \cite[Theorem 2.5]{AlgPersI}, they discuss a lowpass mask $\tau$ for which there is no sos representation for $f(\tau;\cdot),$ meaning that the construction there fails. One way of interpreting this new observation is that when $f(\tau;\cdot)$ is nonnegative, but fails to have an sos representation, we are still able to show that there is a tight wavelet frame based in the multiresolution analysis generated by the refinable function $\phi$ associated with $\tau$.

In Section~\ref{S:slp}, we show that we may once again interpret these results in terms of scaling an ``oblique Laplacian pyramid'' matrix, clearly demonstrating how different assumptions on the factorization of $S(2\omega)+S(2\omega)^{2}f(S,\tau;\omega)$ lead to constructions with just a modified lowpass mask, as in \cite{SLP}; the original lowpass mask, and a collection of highpass masks corresponding to these sums of squares generators, as in \cite{AlgPersI,LaiStock}; or a combination of the two, as discussed above. Moreover, this process of modifying the lowpass mask turns lowpass masks which do not satisfy the oblique sub-QMF condition with the given $S$ into new lowpass masks which do satisfy this condition.\\

The rest of the paper is organized as follows:

In Section~\ref{S:prelim}, we review sums of squares representations and some related literature, as well as some basic facts about tight wavelet frames, and the oblique extension principle. In Section~\ref{S:oep}, we prove Theorems~\ref{th:osqmf} and \ref{th:osqmftwf}. In Section~\ref{S:slp}, we discuss the scaling matrix interpretation of our OEP construction, and show how different factorizations of the scaling matrix lead to different tight wavelet frame constructions, as we did in the introduction for the UEP setting. In Section~\ref{S:boxspl}, we apply the constructions of Section~\ref{S:oep} in the case that the lowpass mask corresponds to a box spline refinable function, obtaining tight wavelet frames with near-maximum vanishing moments for any such lowpass mask, provided that certain univariate trigonometric polynomials may be constructed. We give clear examples of how letting go of a single vanishing moment can significantly decrease the number of highpass masks: see Example~\ref{ex:fewmasks}; and of how choosing a simple form for $S$ can allow us to generalize constructions to arbitrary dimensions easily, while guaranteeing a certain minimum number of vanishing moments for all the wavelets: see Example~\ref{ex:fewmasksextend}.

\section{Preliminaries}
\label{S:prelim}

We denote by $\CC[\z]=\CC[z_{1},\ldots,z_{n}]$ the ring of polynomials in the variables $z_{1},\ldots,z_{n}$ with coefficients in $\CC,$ and similarly for $\RR[\z]$. We denote by $\K=\CC(z_{1},\ldots,z_{n}),$ or the field of fractions of $\CC[\z],$ which is the space of rational polynomials $f=p/q$ where $p,q\in\CC[\z]$ and $q\neq0$, and similarly for $\Kr$. We will also refer to Laurent polynomials, which are the elements of $\K$ of the form $z_{1}^{k_{1}}\cdots z_{n}^{k_{n}}p,$ where $p\in\CC[\z]$ and $k_{1},\ldots,k_{n}\in\mathbb{Z}$. We will denote the set of Laurent polynomials as $\Laur$, or $\RR[\z^{\pm 1}]$ when the coefficient field is $\RR$.

\begin{defn}
We say that a polynomial or Laurent polynomial $f\in\Laur$ has a \emph{sum of hermitian squares representation (or sos representation) on $\Omega\subseteq\CC^{n}$ with $\{g_{j}\}_{j=1}^{J}\subset\CC[\z],$ $J<+\infty$,} when \begin{equation}\label{eq:sosdef}f(z)=\sum_{j=1}^{J}|g_{j}(z)|^{2}\quad\forall z\in\Omega.\end{equation} We may abbreviate this when the set $\Omega$ is clear (we will typically consider $\RR^{n}$ or $(\tor)^{n}$), or just say that $f$ has an sos representation or is an sos on $\Omega$ if there is some finite collection $\{g_{j}\}_{j=1}^{J}\subset\CC[\z]$ for which this equation holds for all $z\in\Omega$.

Similarly, we say that $f\in\K$ has a \emph{sum of rational hermitian squares representation (or sors representation) on $\Omega\subseteq\CC^{n}$ with $\{g_{j}\}_{j=1}^{J}\subset\K,$} if Equation~(\ref{eq:sosdef}) holds for all $z\in\Omega$ at which $f(z)$ is defined.

Using the natural identification $z_{j}=e^{i\omega_{j}}$, we may also apply these definitions in the case that $f$ and the $g_{j}$ are trigonometric polynomials, or rational trigonometric polynomials, where in this case, the natural domain to consider will be $\Omega=\TT^{n}=[-\pi,\pi]^{n}$. \hfill$\square$
\end{defn}

We now collect several related results from the literature, many of which we will be using in the sequel. This first list of results concern trigonometric polynomials.

\begin{res}
\label{r:soslems}
\begin{enumerate}[(a)]
\item (Fej\'{e}r-Riesz \cite{Daub}) Let $f\in\CC[z_1^{\pm 1}]$ be such that $f(z)\geq0$ for all $z\in\tor$. Then $f(z)=|p(z)|^{2}$ for all $z\in\tor,$ where $p\in \CC[z_1]$.
\item (Scheiderer \cite{ScheidRAS}) Let $f\in \CC[z_{1}^{\pm1},z_{2}^{\pm1}]$ be such that $f(z)\geq0$ for all $z\in(\tor)^{2}$. Then $f$ has an sos representation on $(\tor)^{2}$.
\item (Charina et al.\cite{AlgPersI}) Let $n\geq 3$. There is a nonnegative trigonometric polynomial in $n$ variables with no sos representation.
\item (Dritschel \cite{Dritsch}) Let $n\geq 2$, and $f\in\CC[\z^{\pm1}]$ be such that $f(z)>0$ for all $z\in(\tor)^{n}$. Then $f$ has an sos representation on $(\tor)^{n}$.
\end{enumerate}
\end{res}
\begin{proof}
(a) is proved for a special case in \cite{Daub}. (b) may be found as Corollary 3.4 in \cite{ScheidRAS} (see also \cite[Theorem 2.4]{AlgPersI}). (c) is shown in \cite{AlgPersI}. (d) comes from \cite{Dritsch}, but the statement here comes from \cite{GerLai}.
\end{proof}

Moving away from the trigonometric polynomial case to that of ordinary polynomials with real coefficients, we have the following results, the first of which comes from Artin in 1927 \cite{Art} (see also \cite{Beno}), but translated and adapted to our notation. The second comes from Pfister in 1967 \cite{Pfister} (see also \cite{Beno}), and gives a bound on the number of squares in the representation guaranteed by Artin's Theorem. We will use these theorems and the following corollary in the proof of our main theorem.

\begin{res}
\label{r:soslemsrx}
\begin{enumerate}[(a)]
\item (Artin \cite{Art}) Let $f\in\RR[\z]$ be such that $f(z)\geq0$ for all $z\in\RR^{n}$. Then $f$ is an sors on $\RR^{n}$ of functions in $\RR(\z)$, i.e., there exists $\{g_{j}\}_{j=1}^{J}\subset \RR(\z), J<+\infty$ such that for all $z\in\RR^{n}$, $f(z)=\sum_{j=1}^{J} g_{j}(z)^{2}.$
\item (Pfister \cite{Pfister}) Let $f\in \RR[\z]$ be such that $f(z)\geq0$ for all $z\in\RR^{n}$. Then $f$ is an sors on $\RR^{n}$ of at most $2^n$ functions in $\RR(\z)$, i.e., there exist $\{g_{j}\}_{j=1}^{J}$ satisfying the conclusion of part (a) with $J\leq 2^{n}$.
\item Let $f\in\RR(\z)$, $f(z)\geq0$ for all $z\in\RR^{n}$ at which it is defined. Then $f$ is an sors on $\RR^{n}$ of at most $2^{n}$ functions in $\RR(\z)$.
\end{enumerate}
\end{res}

In fact, this last result also holds when the function is a trigonometric polynomial in $n$ variables:
\begin{res}
\label{c:trigsors}
Let $f$ be a nonnegative rational trigonometric polynomial in $n$ variables. There exist at most $2^{n}$ rational trigonometric polynomials $g_{j}$ such that $$f(\omega)=\sum_{j=1}^{J}|g_{j}(\omega)|^{2}\quad\text{ for all }\omega\in\TT^{n}.$$
\end{res}

While it appears that this result has not appeared in the wavelet literature, it may be obtained by combining Result~\ref{r:soslemsrx}(c) with the map $\omega_{j}=2\arctan(x_{j}), 1\leq j\leq n$, which induces $$\cos(\omega_{j})=\frac{1-x_{j}^{2}}{1+x_{j}^{2}},\quad \sin(\omega_{j})=\frac{2x_{j}}{1+x_{j}^{2}}.$$ A version of this map appeared first in \cite{Rudin}, and versions have also appeared in \cite{Dritsch,Parrilo}. The angle addition formulas for sine and cosine ensure that $f$ may be written as a rational polynomial in the variables $\{\cos(\omega_{j}),\sin(\omega_{j})\}_{j=1}^{n}$, and since $f$ is real-valued, this must have real coefficients. Applying this map then results in a multivariable polynomial with real coefficients which is nonnegative on $\RR^{n}$, and this has a rational sum of squares representation by Result~\ref{r:soslemsrx}(c). Inverting this map via $x_{j}=\tan(\omega_{j}/2)$ gives a sum of squares representation for $f$ with rational trigonometric polynomial generators.

Below, we show how this argument works in a particular example.

\begin{exmp}
\label{ex:trigsors}
Consider the trigonometric polynomial $f(\omega)=1-\cos^{2}(\omega_{1})\cos^{2}(\omega_{2}),$ which is nonnegative for all $\omega_{1},\omega_{2}\in\mathbb{R}$. Applying the transformation $\omega_{j}=2\mathrm{arctan}(x_{j})$ to obtain $M(f)$, we get $$M(f)(x)=1-\left(\frac{1-x_{1}^{2}}{1+x_{1}^{2}}\cdot \frac{1-x_{2}^{2}}{1+x_{2}^{2}}\right)^{2}.$$
By Result~\ref{r:soslemsrx}(c), this has an sors representation on $\RR^{2},$ and indeed, $$M(f)(x)=\left(\frac{2x_{1}}{1+x_{1}^{2}}\right)^{2}+\left(\frac{1-x_{1}^{2}}{1+x_{1}^{2}}\right)^{2}\left(\frac{2x_{2}}{1+x_{2}^{2}}\right)^{2}.$$
Inverting the transformation on the variables, we get $$f(\omega)=\sin^{2}(\omega_{1})+\cos^{2}(\omega_{1})\sin^{2}(\omega_{2}).$$\hfill$\square$
\end{exmp}

\subsection{Tight Wavelet Frames}
\label{s:twfs}

Let $\dmat\in M_{n}(\mathbb{Z})$ be a dilation matrix, so that the set of eigenvalues of $\dmat$ lies in $\{z\in\mathbb{C}:|z|>1\}$. Let $\ddet=|\mathrm{det}(\dmat)|$. Let $\Gamma$ be a complete set of distinct coset representatives of $\ZZ^{n}/\dmat\ZZ^{n}$ containing 0, and let $\Gamma^{*}$ be a complete set of distinct coset representatives of $(2\pi\dmat^{-T}\ZZ^{n})/(2\pi\ZZ^{n})$ containing 0. In the introduction, we considered only the case of $\dmat=2I$, and chose $\Gamma=\{0,1\}^{n}$, $\Gamma^{*}=\{0,\pi\}^{n}$, but we now consider the general case.

\begin{defn}
Given a lowpass mask $\tau,$ which is a trigonometric polynomial satisfying $\tau(0)=1,$ the \emph{refinable function} $\phi$ \emph{associated with }$\tau$ satisfies $\hat{\phi}(\dmat^{T}\omega)=\tau(\omega)\hat{\phi}(\omega)$ for all $\omega\in\RR^{n}$. It may be defined by its Fourier transform as $\hat{\phi}(\omega)=\prod_{j=1}^{\infty}\tau((\dmat^{-T})^{j}\omega)$ for all $\omega\in\mathbb{R}^{n}$. Given a collection of highpass masks $q_{\ell},\,1\leq \ell\leq r$ which are rational trigonometric polynomials satisfying $q_{\ell}(0)=0,$ we define $(\tau,q_{1},\ldots,q_{r})$ as the \emph{combined MRA mask}, and \emph{the wavelet system defined by }$(\tau,q_{1},\ldots,q_{r})$ is then the set $\{\psi_{j,k}^{(\ell)}:j\in\mathbb{Z},k\in\mathbb{Z}^{n},1\leq\ell\leq r\}$, where $\psi_{j,k}^{(\ell)}=\ddet^{j/2}\psi^{(\ell)}(\dmat^j \cdot-k),$ and $\psi^{(\ell)}$ is defined by its Fourier transform as $\hat{\psi}^{(\ell)}(\dmat^{T}\omega)=q_{\ell}(\omega)\hat{\phi}(\omega)$ for all $\omega\in\RR^{n}$.\hfill$\square$
\end{defn}

When a wavelet system has the property that $\sum_{j,k,\ell}|\langle f,\psi_{j,k}^{(\ell)}\rangle|^{2}=\|f\|_{L^{2}(\RR^{n})}^{2}$ for all $f\in L^{2}(\RR^{n}),$ we say that it is a \emph{tight wavelet frame (TWF)}, and in this case $f=\sum_{j,k,\ell}\langle f,\psi_{j,k}^{(\ell)}\rangle\psi_{j,k}^{(\ell)}$. This expansion is similar to the one given by an orthonormal basis for $L^{2}(\RR^{n}),$ but the wavelets $\psi_{j,k}^{(\ell)}$ may not be orthogonal in the case that they only form a TWF. When in addition $\{\psi_{j,k}^{(\ell)}\}$ is an orthonormal set, we say that the wavelet system is an \emph{orthonormal wavelet basis}. For more details on these ideas, see \cite{Daub}.

The statement of the OEP is given below, from \cite{DHRS}. This uses the notation $\sigma(\phi):=\{\omega\in\TT^{n}:\hat{\phi}(\omega+2\pi k)\neq0,\text{ for some }k\in\ZZ^{n}\},$ as well as $\delta:\Gamma^{*}\rightarrow\{0,1\},$ which always takes the value zero except for $\delta(0)=1$. We have adapted the notation of this theorem to our setting.

\begin{res}
\label{res:oep}
Let $\tau,q_{1},\ldots,q_{r}$ be $2\pi$-periodic functions. Suppose that
\begin{enumerate}[(a)]
\item Each function $\tau,q_{j}$ belongs to $L^{\infty}(\TT^{n})$.
\item The refinable function $\phi$ satisfies $\lim_{\omega\rightarrow0}\hat{\phi}(\omega)=1.$
\item The function $[\hat{\phi},\hat{\phi}]:=\sum_{k\in\ZZ^{n}}|\hat{\phi}(\cdot+2\pi k)|^{2}$ belongs to $L^{\infty}(\TT^{n})$.
\end{enumerate}
Suppose there exists a $2\pi$-periodic function $S$ that satisfies the following:
\begin{enumerate}[(i)]
\item $S\in L^{\infty}(\TT^{n})$ is nonnegative, continuous at the origin, and $S(0)=1$.
\item If $\omega\in\sigma(\phi)$, and if $\gamma\in\Gamma^{*}$ is such that $\omega+\gamma\in\sigma(\phi),$ then $$S(\dmat^{T}\omega)\tau(\omega)\overline{\tau(\omega+\gamma)}+\sum_{j=1}^{r}q_{j}(\omega)\overline{q_{j}(\omega+\gamma)}=S(\omega)\delta(\gamma).$$
\end{enumerate}
Then the wavelet system defined by $(\tau,q_{1},\ldots,q_{r})$ is a tight wavelet frame for $L^{2}(\RR^{n})$.
\end{res}

In what follows, we will always be considering $\tau$ as a trigonometric polynomial lowpass mask, which corresponds to a compactly supported refinable function. In this case, $\sigma(\phi)=\TT^{n}$ (see \cite{DHRS}). We also note that when $\phi$ satisfies condition (c), it necessarily belongs to $L^{2}(\RR^{n})$. The function $S$ is called the \emph{vanishing moment recovery (vmr) function}, since the additional flexibility it brings may be used to construct highpass masks with better vanishing moments than in the unitary extension principle setting, which forces $S\equiv 1$.

\section{OEP Tight Wavelet Frames from Sors Representations}
\label{S:oep}

Given a trigonometric polynomial lowpass mask $\tau$ and rational trigonometric polynomial vmr function $S$, we now give a condition on this pair which guarantees the existence of rational trigonometric polynomial highpass masks $q_{\ell}$ satisfying the OEP conditions (specifically (ii) above), and we show how to find these masks constructively. The condition we impose might be thought of as an oblique extension of the well-known sub-QMF condition, and indeed when $S\equiv 1,$ the ``oblique sub-QMF condition'' reduces to the sub-QMF condition, which is necessary for constructing a tight wavelet frame with the UEP \cite{AlgPersI}. Analogously, the OEP conditions will necessitate our oblique version. We begin by stating the main theorems, which are split into an algebraic part and an analytical part. The proofs of these theorems then follow, along with a few requisite definitions and lemmata.

\begin{thm}
\label{th:osqmf}
Let $S$ be a nonzero rational trigonometric polynomial which is nonnegative on $\TT^{n}$, and let $\tau$ be a trigonometric polynomial lowpass mask. The following are equivalent:
\begin{enumerate}[(A)]
\item The Oblique sub-QMF condition holds: \begin{equation}\label{eq:obsubqmf} \sum_{\gamma\in\Gamma^{*}}\frac{|\tau(\omega+\gamma)|^{2}}{S(\omega+\gamma)}\leq\frac{1}{S(\dmat^{T}\omega)}\quad\text{for all }\omega\in\mathbb{T}^{n}\text{ at which both sides are defined},\end{equation}
\item There exist rational trigonometric polynomials $\{q_{\ell}\}_{\ell=1}^{r}$ such that for all $\gamma\in\Gamma^{*}$ and $\omega\in\TT^{n}$ at which both sides are defined: \begin{equation}\label{eq:oepfull}S(\dmat^{T}\omega)\tau(\omega)\overline{\tau(\omega+\gamma)}+\sum_{\ell=1}^{r}q_{\ell}(\omega)\overline{q_{\ell}(\omega+\gamma)}=\begin{cases}S(\omega)&\text{if }\gamma=0\\0&\text{otherwise.}\end{cases}\end{equation}
\end{enumerate}
Moreover, provided that either (A) or (B) holds, there exist (a potentially different set of) rational trigonometric polynomials $\{q_{\ell}\}_{\ell=1}^{r}$ such that for all $\gamma\in\Gamma^{*}$ and $\omega\in\TT^{n}$ at which both sides are defined, Equation~(\ref{eq:oepfull}) holds, where $r\leq 2^{n}(1+\ddet)$.
\end{thm}

The next theorem combines these conditions with an additional assumption on the rational trigonometric polynomial $S$, guaranteeing that these masks generate a TWF. This might be seen as an extension of \cite[Lemma 2.1]{HanOrtho}.

\begin{thm}
\label{th:osqmftwf}
Assume the setting of the previous theorem. Suppose, in addition to satisfying one of (A) or (B), that $S$ is continuous at $0$ with $S(0)=1,$ and that $S$ and $1/S$ belong to $L^{\infty}(\TT^{n})$. Then the wavelet system defined by the combined MRA mask $(\tau,q_{1},\ldots,q_{r})$ is a tight wavelet frame. 
\end{thm}

The assumptions that $S,1/S\in L^{\infty}(\TT^{n})$ guarantee that $S$ and $1/S$ are pole-free, so in this setting, Equations~(\ref{eq:obsubqmf}) and (\ref{eq:oepfull}) hold for all $\omega\in\TT^{n}$. Now, we turn to the proofs of these theorems.

We begin with some definitions and basic lemmas. Recall that $\dmat\in M_{n}(\ZZ)$ is a dilation matrix.

\begin{defn}
We define the group action of $\grp=(2\pi\dmat^{-T}\ZZ^{n})/(2\pi\ZZ^{n})$ on a rational trigonometric polynomial in the following way: for $\gamma\in\Gamma^{*}$ and $f$ a rational trigonometric polynomial, $\gamma:f \mapsto f^{\gamma}(\cdot)=f(\cdot+\gamma).$ Note that this group action is independent of the set $\Gamma^{*}$ of coset representatives for $\grp$.

We say that a rational trigonometric polynomial $f$ is $\grp$-invariant if for all $\gamma\in\Gamma^{*}$, $f^{\gamma}=f$.

We say that $H(\omega)$ is a $\grp$-vector for the rational trigonometric polynomial $h$ if $H(\omega)=[h^{\gamma}(\omega)]_{\gamma\in\Gamma^{*}}.$ We also call a vector a $\grp$-vector if it is of this form for some rational trigonometric polynomial.\hfill$\square$
\end{defn}

\begin{defn}
For $f$ a rational trigonometric polynomial, we define its polyphase components $f_{\nu}$, $\nu\in\Gamma$, by \begin{equation}\label{eq:poly}f_{\nu}(\dmat^{T}\omega)=\ddet^{-1/2}\sum_{\gamma\in\Gamma^{*}}f^{\gamma}(\omega)e^{-i(\omega+\gamma)\cdot\nu}.\end{equation} \hfill$\square$
\end{defn}
We will also use the following dual relation, which is easy to show using the definition above:
\begin{equation}\label{eq:dualpoly}f(\omega)=\ddet^{-1/2}\sum_{\nu\in\Gamma}f_{\nu}(\dmat^{T}\omega)e^{i\omega\cdot\nu}.\end{equation}

Note that if $f$ is a rational trigonometric polynomial, then $f(\dmat^{T}\cdot)$ is $\grp$-invariant, since $(f(\dmat^{T}\cdot))^{\gamma}(\omega)=f(\dmat^{T}(\omega+\gamma))=f(\dmat^{T}\omega),$ because $\dmat^{T}\gamma\in 2\pi\ZZ^{n}.$ If $f$ is $\grp$-invariant, then $f(\omega)=\ddet^{-1}\sum_{\gamma\in\Gamma^{*}}f^{\gamma}(\omega)=\ddet^{-1/2}f_{0}(\dmat^{T}\omega),$ so $f$ is a rational trigonometric polynomial in $\dmat^{T}\omega$.

\begin{defn} 
Let $f$ be a rational trigonometric polynomial with an sos or sors of functions $g_{j},1\leq j\leq J$. We say that $f$ has a \emph{$\grp$-invariant so(r)s} if its so(r)s representation has the property that $g_{j}^{\gamma}=g_{j}$ for all $\gamma\in\Gamma^{*}$, for all $1\leq j\leq J$.\hfill$\square$
\end{defn}

We now prove a few lemmas regarding $\grp$-invariance and sums of squares representations. Similar results appear in  \cite[Lemma 2.1]{AlgPersI} under more restrictive assumptions.

\begin{lem}
\label{l:ginv}
Let $f$ be a rational trigonometric polynomial.
\begin{enumerate}[(a)]
\item $\sum_{\gamma\in\Gamma^{*}}|f^{\gamma}|^{2}=\sum_{\nu\in\Gamma}|f_{\nu}(\dmat^{T}\cdot)|^{2}.$
\item If $f$ is $\grp$-invariant, then it is an so(r)s if and only if it is a $\grp$-invariant so(r)s.
\end{enumerate}
\end{lem}
\begin{proof}
(a) We use Equation~(\ref{eq:dualpoly}) to compute, for $\omega\in\TT^{n}$ such that the left hand side is defined: $$\sum_{\gamma\in\Gamma^{*}}|f^{\gamma}(\omega)|^{2}=\ddet^{-1}\sum_{\gamma\in\Gamma^{*}}\sum_{\nu,\nu'\in\Gamma}f_{\nu}(\dmat^{T}\omega)\overline{f_{\nu'}(\dmat^{T}\omega)}e^{i(\omega+\gamma)\cdot(\nu-\nu')}=\sum_{\nu\in\Gamma}|f_{\nu}(\dmat^{T}\omega)|^{2},$$ using the fact that for $k\in\ZZ^{n}$, $\sum_{\gamma\in\Gamma^{*}}e^{i\gamma\cdot k}$ is equal to $\ddet$ when $k\equiv0\pmod{\dmat\ZZ^{n}}$ and is 0 otherwise.

(b) The converse is obvious, so if $f=\sum_{j=1}^{J}|g_{j}|^{2},$ then for all $\omega\in\TT^{n}$ where it is defined, $$f(\omega)=\ddet^{-1}\sum_{\gamma\in\Gamma^{*}}f^{\gamma}(\omega)=\ddet^{-1}\sum_{j=1}^{J}\sum_{\gamma\in\Gamma^{*}}|g_{j}^{\gamma}(\omega)|^{2}=\ddet^{-1}\sum_{j=1}^{J}\sum_{\nu\in\Gamma}|(g_{j})_{\nu}(\dmat^{T}\omega)|^{2},$$ the last equality following from part (a) which was just proved. The last expression is clearly a sum of $\grp$-invariant squares, which completes the proof.
\end{proof}

From the proof of (b), it may appear that the number of sors generators will increase by a factor of $\ddet$ when we require them to be $\grp$-invariant. However, if $f$ is $\grp$-invariant and an sors, then $f=g(\dmat^{T}\cdot)$, where $g$ is certainly nonnegative. By Corollary~\ref{c:trigsors}, $g$ has an sors $g=\sum_{j=1}^{J}|g_{j}|^{2},$ where $J\leq 2^{n}$, which means that $f=\sum_{j=1}^{J}|g_{j}(\dmat^{T}\cdot)|^{2}$. On the other hand, when we want an \emph{sos} representation, this argument fails because Corollary~\ref{c:trigsors} may introduce denominators. As such, the $\grp$-invariant sos representation may require at most $\ddet$ times as many generators as the original sos.

The following lemma combines Corollary~\ref{c:trigsors} with an idea from the proof of \cite[Theorem 6.1]{LaiStock}.

\begin{lem}
\label{l:siginvsors}
Suppose $S$ is a nonzero rational trigonometric polynomial such that $S(\omega)\geq0$ for all $\omega\in\TT^{n}$ where it is defined. Let $\Sigma(\omega)=\mathrm{diag}(S(\omega+\gamma))_{\gamma\in\Gamma^{*}}$ (for some ordering of $\Gamma^{*}$). Then $\Sigma^{-1}(\omega)=A(\omega)A(\omega)^{*}$ for all $\omega\in\TT^{n}$ where $\Sigma^{-1}(\omega)$ is defined, and $A(\omega)$ is a $\ddet\times M$ matrix with rational trigonometric polynomial entries such that each column of $A(\omega)$ is a $\grp$-vector, with $M\leq 2^{n}\ddet$.
\end{lem}

\begin{proof}
By Corollary~\ref{c:trigsors}, there are rational trigonometric polynomials $s_{j}$ such that $1/S(\omega)=\sum_{j=1}^{J}|s_{j}(\omega)|^{2}$ for all $\omega\in\TT^{n}$ where $1/S(\omega)$ is defined, and $J\leq 2^{n}$. Let $\mathbf{s}(\omega)=[s_{1}(\omega),s_{2}(\omega),\ldots,s_{J}(\omega)]$. Let $$A(\omega)=\ddet^{-1/2}[e^{i\nu\cdot(\omega+\gamma)}\mathbf{s}(\omega+\gamma)]_{\gamma\in\Gamma^{*},(\nu,j)\in\Gamma\times\{1,\ldots,J\}},$$ which is a $\ddet\times (\ddet J)$ matrix with rational trigonometric polynomial entries. Moreover, given $(\nu,j)\in\Gamma\times\{1,\ldots,J\},$ if we let $a_{(\nu,j)}(\omega)=\ddet^{-1/2}e^{i\nu\cdot\omega}s_{j}(\omega)$, we see that $A(\omega)_{\gamma,(\nu,j)}=a_{(\nu,j)}(\omega+\gamma)$, so $A(\omega)$ has $\grp$-vector columns. Then for all $\gamma,\gamma'\in\Gamma^{*}$, letting $\delta:2\pi\dmat^{-T}\ZZ^{n}\rightarrow\{0,1\}$ always take value zero except $\delta(0)=1$, \begin{align*}(A(\omega)A(\omega)^{*})_{\gamma,\gamma'}&=\ddet^{-1}\sum_{j=1}^{J}s_{j}(\omega+\gamma)\overline{s_{j}(\omega+\gamma')}\sum_{\nu\in\Gamma}e^{i\nu\cdot(\gamma-\gamma')}\\
&=\delta(\gamma-\gamma')\sum_{j=1}^{J}|s_{j}(\omega+\gamma)|^{2}\\
&=\delta(\gamma-\gamma')\frac{1}{S(\omega+\gamma)}=\Sigma^{-1}(\omega)_{\gamma,\gamma'}.\end{align*} This clearly holds wherever $\Sigma^{-1}(\omega)$ is defined.
\end{proof}

We are now ready to prove Theorem~\ref{th:osqmf}.

\emph{Proof of Theorem~\ref{th:osqmf}:} To avoid excessive verbiage, throughout this proof, all equalities should be taken to hold wherever both sides are defined, which because we are considering a finite collection of rational trigonometric polynomials, is an open, dense set with full measure.

\emph{(i) Proof that A implies B:} Suppose Statement A of the theorem. We observe that by Corollary~\ref{c:trigsors} and Lemma~\ref{l:ginv}(b), for $\omega\in\TT^{n}$, \begin{equation}\label{eq:oepfsors}\frac{1}{S(\dmat^{T}\omega)}-\sum_{\gamma\in\Gamma^{*}}\frac{|\tau(\omega+\gamma)|^{2}}{S(\omega+\gamma)}=\sum_{j=1}^{J}|g_{j}(\dmat^{T}\omega)|^{2},\end{equation} where $g_{j}$ are rational trigonometric polynomials, since the left hand side is nonnegative by Equation~(\ref{eq:obsubqmf}) and is clearly $\grp$-invariant. Moreover, from the argument after Lemma~\ref{l:ginv}, $J$ in Equation~(\ref{eq:oepfsors}) is no greater than $2^{n}$. Let $H(\omega)=[\tau(\omega+\gamma)]_{\gamma\in\Gamma^{*}}$ be a column vector, and let $G(\omega)=[g_{j}(\omega)]_{j=1}^{J}$. Recall that $\Sigma(\omega)=\mathrm{diag}(S(\omega+\gamma))_{\gamma\in\Gamma^{*}}.$ We observe that \begin{align*}&[H(\omega)\;S(\dmat^{T}\omega)H(\omega)G(\dmat^{T}\omega)^{*}\;\Sigma(\omega)-S(\dmat^{T}\omega)H(\omega)H(\omega)^{*}]\\&\quad\times \begin{bmatrix} S(\dmat^{T}\omega)&0&0\\0&I_{J}&0\\0&0&\Sigma^{-1}(\omega)\end{bmatrix}\begin{bmatrix} H(\omega)^{*}\\ S(\dmat^{T}\omega)G(\dmat^{T}\omega)H(\omega)^{*}\\\Sigma(\omega)-S(\dmat^{T}\omega)H(\omega)H(\omega)^{*}\end{bmatrix}\\
&=S(\dmat^{T}\omega)H(\omega)H(\omega)^{*}+\left[S(\dmat^{T}\omega)^{2}G(\dmat^{T}\omega)^{*}G(\dmat^{T}\omega)\right]H(\omega)H(\omega)^{*}\\ &\quad+\Sigma(\omega)-\left[2S(\dmat^{T}\omega)\right]H(\omega)H(\omega)^{*}+\left[S(\dmat^{T}\omega)^{2}H(\omega)^{*}\Sigma^{-1}(\omega)H(\omega)\right]H(\omega)H(\omega)^{*}\\&=\Sigma(\omega)-S(\dmat^{T}\omega)H(\omega)H(\omega)^{*}\\ &\quad+S(\dmat^{T}\omega)^2\left[G(\dmat^{T}\omega)^{*}G(\dmat^{T}\omega)+H(\omega)^{*}\Sigma^{-1}(\omega)H(\omega)\right]H(\omega)H(\omega)^{*}\\&=\Sigma(\omega),\end{align*}
where the last equation follows by Equation~(\ref{eq:oepfsors}).
The columns of $S(\dmat^{T}\omega)H(\omega)G(\dmat^{T}\omega)^{*}$ are $\grp$-vectors, since $H(\omega)$ is and the other factors are $\grp$-invariant. Now we use Lemma~\ref{l:siginvsors} to see that $\Sigma(\omega)^{-1}=A(\omega)A(\omega)^{*},$ where the columns of $A$ are $\grp$-vectors. We observe that for any rational trigonometric polynomial $g$: $$(\Sigma(\omega)-S(\dmat^{T}\omega)H(\omega)H(\omega)^{*})[g^{\gamma}(\omega)]_{\gamma\in\Gamma^{*}} =\left[S^{\gamma}(\omega)g^{\gamma}(\omega)-S(\dmat^{T}\omega)\tau^{\gamma}(\omega)\sum_{\gamma'\in\Gamma^{*}} \overline{\tau^{\gamma'}(\omega)}g^{\gamma'}(\omega)\right]_{\gamma\in\Gamma^{*}}$$ So we see that the columns of $(\Sigma(\omega)-S(\dmat^{T}\omega)H(\omega)H(\omega)^{*})A(\omega)$ are $\grp$-vectors, and $A$ has $M\leq 2^{n}\ddet$ columns. Then the following rational trigonometric polynomials are defined and satisfy Equation~(\ref{eq:oepfull}) with $\tau$:
\begin{align*}
q_{1,j}(\omega)&=S(\dmat^{T}\omega)\tau(\omega)\overline{g_j(\dmat^{T}\omega)}&1\leq j\leq J\leq 2^{n},\\
q_{2,m}(\omega)&=S(\omega)a_{m}(\omega)-S(\dmat^{T}\omega)\tau(\omega)\sum_{\gamma\in\Gamma^{*}}\overline{\tau(\omega+\gamma)}a_{m}(\omega+\gamma)&1\leq m\leq M\leq 2^{n}\ddet,
\end{align*}
where the $m$th column of $A(\omega)$ is a $\grp$-vector for the rational trigonometric polynomial $a_{m}(\omega),$ $1\leq m\leq M$. This construction also gives the bound $r\leq 2^{n}(1+\ddet).$\\

\emph{(ii) Proof that B implies A:} Suppose Statement B of the theorem. Then for $\omega\in\TT^{n}$, $$\Sigma(\omega)-S(\dmat^{T} \omega)H(\omega)H(\omega)^{*}=Q(\omega)Q(\omega)^{*},$$ where $Q(\omega)$ is a $\ddet\times r$ matrix with rational trigonometric polynomial entries, and columns of the form $[q_{\ell}(\omega+\gamma)]_{\gamma\in\Gamma^{*}}$. Taking the determinant on both sides of the previous equation, the left hand gives 
 \begin{align*}\mathrm{det}(\Sigma(\omega)-S(\dmat^{T}\omega)H(\omega)H(\omega)^{*})&=\mathrm{det}(\Sigma^{1/2}(\omega)(I-S(\dmat^{T}\omega)\Sigma^{-1/2}(\omega)H(\omega)H(\omega)^{*}\Sigma^{-1/2}(\omega))\Sigma^{1/2}(\omega))\\&=\mathrm{det}(\Sigma(\omega))(1-S(\dmat^{T}\omega)H(\omega)^{*}\Sigma^{-1}(\omega)H(\omega)).\end{align*} Since $Q(\omega)Q(\omega)^{*}$ is positive semidefinite for all $\omega\in\TT^{n}$, its determinant is nonnegative. $S$ is nonnegative, so for $\omega\in\TT^{n}$, $1/\mathrm{det}(\Sigma(\omega))\geq0,$ and $1/S(\mathcal{M}^{T}\omega)\geq0$. Then Equation~(\ref{eq:obsubqmf}) follows, since $$\frac{1}{S(\dmat^{T}\omega)}-\sum_{\gamma\in\Gamma^{*}}\frac{|\tau(\omega+\gamma)|^{2}}{S(\omega+\gamma)}=\frac{\mathrm{det}(Q(\omega)Q(\omega)^{*})}{\mathrm{det}(\Sigma(\omega))S(\mathcal{M}^{T}\omega)}.$$\hfill$\square$
  
Now we turn to the analytical part of the construction, with the proof of Theorem~\ref{th:osqmftwf}. We will see that the additional conditions on the vmr function $S$ are used in order to guarantee the ess. boundedness of the constructed highpass masks, as well as that of $[\hat{\phi},\hat{\phi}]$, as required in Result~\ref{res:oep}(a) and (c).
 
 \emph{Proof of Theorem~\ref{th:osqmftwf}:} We seek to apply Result~\ref{res:oep}. It is clear from the assumptions that (i) and (ii) hold, since the conditions on $S$ mean that (\ref{eq:oepfull}) holds for all $\omega\in\TT^{n}$. Since $\tau$ is a trigonometric polynomial, it is continuous and therefore bounded; because its corresponding filter has finite support, $\phi$ is compactly supported, so $\hat{\phi}$ is continuous, and (b) also holds. Rearranging Equation~(\ref{eq:oepfull}) and looking at the case $\gamma=0,$ we see that $\sum_{\ell=1}^{r}|q_{\ell}(\omega)|^{2}=S(\omega)-S(\mathcal{M}^{T}\omega)|\tau(\omega)|^{2},$ so the ess. boundedness of the right hand side implies the ess. boundedness of $q_{\ell}$ for all $1\leq \ell\leq r$. This proves (a), so it remains to show (c).

We argue as in the proof of the first half of \cite[Lemma 2.1]{HanOrtho}. Recall that $\hat{\phi}(\omega):=\prod_{j=1}^{\infty}\tau((\dmat^{-T})^{j}\omega)$ for all $\omega\in\RR^{n}$. Let \begin{align}f_{0}(\omega)&:=\chi_{[-\pi,\pi)^{n}}(\omega)(S(\omega))^{-1/2},\text{ and for all }j\geq1,\text{ let }\notag\\ f_{j}(\omega)&:=\tau(\dmat^{-T}\omega)f_{j-1}(\dmat^{-T}\omega)=\chi_{(\dmat^{T})^{j}[-\pi,\pi)^{n}}(\omega)(S((\dmat^{-T})^{j}\omega))^{-1/2}\prod_{\ell=1}^{j}\tau((\dmat^{-T})^{\ell}\omega).\label{eq:fjdef}\end{align} We now prove by induction that $[f_{j},f_{j}](\omega)=\sum_{k\in\ZZ^{n}}|f_{j}(\omega+2\pi k)|^{2}\leq 1/S(\omega)$. Clearly, $[f_{0},f_{0}](\omega)=1/S(\omega),$ so suppose by way of induction that for some $j-1\geq0,$ $[f_{j-1},f_{j-1}](\omega)\leq 1/S(\omega)$ for all $\omega\in\RR^{n}$. Then
\begin{align}
[f_{j},f_{j}](\omega)&=\sum_{k\in\ZZ^{n}}|\tau(\dmat^{-T}(\omega+2\pi k))|^{2}|f_{j-1}(\dmat^{-T}(\omega+2\pi k))|^{2}\notag\\
&=\sum_{\gamma\in\Gamma^{*}}\sum_{k\in\ZZ^{n}}|\tau(\dmat^{-T}\omega+\gamma+2\pi k)|^{2}|f_{j-1}(\dmat^{-T}\omega+\gamma+2\pi k)|^{2}\notag\\
&=\sum_{\gamma\in\Gamma^{*}}|\tau(\dmat^{-T}\omega+\gamma)|^{2}[f_{j-1},f_{j-1}](\dmat^{-T}\omega+\gamma)\notag\\
&\leq\sum_{\gamma\in\Gamma^{*}}|\tau(\dmat^{-T}\omega+\gamma)|^{2}\frac{1}{S(\dmat^{-T}\omega+\gamma)}\label{eq:bracketbnd1}\\
&\leq\frac{1}{S(\dmat^{T}(\dmat^{-T}\omega))}\label{eq:bracketbnd2}\\
&=\frac{1}{S(\omega)},\notag
\end{align}
where we applied Equation~(\ref{eq:obsubqmf}), which holds for all $\omega\in\TT^{n}$ because of the assumptions on $S$, to obtain the last inequality. We note that as $j\rightarrow\infty,$ using the continuity of $S$ at 0, $f_{j}(\omega)\rightarrow \hat{\phi}(\omega)$. Applying Fatou's Lemma with the counting measure, since $|f_{j}(\omega)|^{2}\geq0$ for all $\omega\in\RR^{n}$ and $j\geq0,$ we see that $$[\hat{\phi},\hat{\phi}](\omega)=\sum_{k\in\ZZ^{n}}|\hat{\phi}(\omega+2\pi k)|^{2}=\sum_{k\in\ZZ^{n}}\lim_{j\rightarrow\infty}|f_{j}(\omega+2\pi k)|^{2}\leq \liminf_{j\rightarrow\infty}\sum_{k\in\ZZ^{n}}|f_{j}(\omega+2\pi k)|^{2}\leq \frac{1}{S(\omega)}\text{ for all }\omega\in\TT^{n},$$ whence applying the ess. boundedness assumption on $1/S$ yields (c), which completes the proof.\hfill$\square$
 
 It is easy to see that under the weaker assumption that $1/S$ is integrable over $[-\pi,\pi)^{n},$ we may not have that all of the masks $S,$ $q_{\ell}$ are ess. bounded, so that (i) or (a) may not hold, but the argument for (c) shows that $$\|\hat{\phi}\|_{2}^{2}=\int_{[-\pi,\pi)^{n}}[\hat{\phi},\hat{\phi}](\omega)\mathrm{d}\omega\leq\int_{[-\pi,\pi)^{n}}\frac{1}{S(\omega)}\mathrm{d}\omega<+\infty,$$ which shows that $\hat{\phi},$ and therefore also $\phi$, belong to $L^{2}(\RR^{n})$.
 
%It is not difficult to show that Equation~(\ref{eq:obsubqmf}) implies that $\Sigma(\omega)-S(\dmat^{T}\omega)H(\omega)H(\omega)^{*}$ is positive semidefinite, whence Theorem~\ref{th:nonnegsorsmatv} yields the existence of $Q$ with $\Sigma(\omega)-S(\dmat^{T}\omega)H(\omega)H(\omega)^{*}=Q(\omega)Q(\omega)^{*}$ immediately, but the explicit construction given above ensures that the matrix $Q$ has $\grp$-vector columns.

Combining Theorems~\ref{th:osqmf} and \ref{th:osqmftwf} yields the following corollary, which applies in the setting that $\tau$ satisfies the ordinary sub-QMF condition.

\begin{cor}
\label{c:sqmf}
Let $\tau$ be a trigonometric polynomial lowpass mask. The following are equivalent:
\begin{enumerate}[(A)]
\item The sub-QMF condition holds: \[\sum_{\gamma\in\Gamma^{*}}|\tau(\omega+\gamma)|^{2}\leq 1\quad\text{for all }\omega\in\mathbb{T}^{n},\]
\item There exist rational trigonometric polynomials $\{q_{\ell}\}_{\ell=1}^{r}$ such that for all $\gamma\in\Gamma^{*}$ and $\omega\in\mathbb{T}^{n}$: \begin{equation}\label{eq:uepfull}\tau(\omega)\overline{\tau(\omega+\gamma)}+\sum_{\ell=1}^{r}q_{\ell}(\omega)\overline{q_{\ell}(\omega+\gamma)}=\begin{cases}1&\text{if }\gamma=0\\0&\text{otherwise.}\end{cases}\end{equation}
\end{enumerate}
Moreover, provided that either (A) or (B) holds, there exist (a potentially different set of) rational trigonometric polynomials $\{q_{\ell}\}_{\ell=1}^{r}$ such that for all $\gamma\in\Gamma^{*}$ and $\omega\in\mathbb{T}^{n}$, Equation~(\ref{eq:uepfull}) holds, with $r\leq 2^{n}(1+\ddet)$. When one of (A) or (B) holds, the wavelet system defined by the combined MRA mask $(\tau,q_{1},\ldots,q_{r})$ is a tight wavelet frame.
\end{cor}

This corollary is quite similar to the result \cite[Thm. 2.2]{AlgPersI}, but both statements here are weaker than the ones that appear in that theorem. In particular, the analogous statement for (A) in \cite{AlgPersI} requires the existence of an sos representation for $1-\sum_{\gamma}|\tau^{\gamma}|^{2}$, but their result guarantees the existence of $q_{\ell}$ which are trigonometric polynomials in (B).

Let us recall a few basic definitions.

\begin{defn} For a lowpass mask $\tau$, the \emph{accuracy number} is defined to be the minimum order of vanishing of $\tau$ at the points $\gamma\in\Gamma^{*}\setminus\{0\}$. For a highpass mask $g$, the \emph{vanishing moments} are the order of vanishing of $g$ at $\omega=0$.
\end{defn}

Inspecting the OEP conditions, we see that we can give a lower bound on the number of vanishing moments of the constructed wavelet system in terms of the accuracy number of $\tau$ and the vanishing moments of $$f(S,\tau;\cdot)=\frac{1}{S(\dmat^{T}\cdot)}-\sum_{\gamma\in\Gamma^{*}}\frac{|\tau^{\gamma}|^{2}}{S^{\gamma}}.$$ A similar discussion for approximation orders is given in \cite{DHRS}, without the formulation involving $f(S,\tau;\cdot)$. The following proposition describes this lower bound.

\begin{prop}
\label{prop:oephpvms}
In Theorem~\ref{th:osqmf}, let $\tau$ have accuracy number $a>0$, $f(S,\tau;\cdot)$ have vanishing moments $m$, and $S-S(\dmat^{T}\cdot)|\tau|^{2}$ have vanishing moments $j$. If (A) or (B) holds in that theorem, then the highpass masks $q_{\ell}, 1\leq \ell\leq r$ in (B) have at least $\lfloor j/2\rfloor\geq\lfloor\min\{a,m/2\}\rfloor$ vanishing moments.
\end{prop}
\begin{proof}
Rearranging the OEP conditions (\ref{eq:oepfull}) with $\gamma=0,$ we get $$S(\omega)-S(\dmat^{T}\omega)|\tau(\omega)|^{2}=\sum_{\ell=1}^{r}|q_{\ell}(\omega)|^{2},$$ so if the left-hand side is $O(|\omega|)^{j}$ for $\omega\approx0,$ then $q_{\ell}=O(|\omega|^{j/2})$ there, for all $1\leq \ell\leq r.$ If $f(S,\tau;\omega)=O(|\omega|^{m})$ for $\omega\approx0$, then $$\frac{1}{S(\dmat^{T}\omega)}-\frac{|\tau(\omega)|^{2}}{S(\omega)}=O(|\omega|^{m})+\sum_{\gamma\in\Gamma^{*}\setminus\{0\}}\frac{|\tau(\omega+\gamma)|^{2}}{S(\omega+\gamma)},$$ which is $O(|\omega|^{\min\{m,2a\}})$ for $\omega\approx0$. Then \begin{align*}S(\omega)-S(\dmat^{T}\omega)|\tau(\omega)|^{2}&=S(\omega)S(\dmat^{T}\omega)\left(\frac{1}{S(\dmat^{T}\omega)}-\frac{|\tau(\omega)|^{2}}{S(\omega)}\right)\\&=S(\omega)S(\dmat^{T}\omega)O(|\omega|^{\min\{m,2a\}})\text{ for }\omega\approx0,\end{align*} which means that $j\geq\min\{m,2a\}$. This completes the proof.
\end{proof}

\begin{rem}
By careful consideration of the proofs of Theorems~\ref{th:osqmf} and \ref{th:osqmftwf}, it is possible to show that for a trigonometric polynomial lowpass mask $\tau$ and rational trigonometric polynomial masks $q_{\ell},1\leq \ell\leq r$, satisfying the OEP conditions~(\ref{eq:oepfull}) with some nonzero, nonnegative $S\in L^{\infty}(\TT^{n})$, then the oblique sub-QMF condition (\ref{eq:obsubqmf}) holds on a large subset of $\TT^{n}$. This inequality may be used to show that when in addition, $1/S\in L^{\infty}(\TT^{n})$, the wavelet system defined by the combined MRA mask $(\tau,q_{1},\ldots,q_{r})$ is a tight wavelet frame, or if $1/S$ is only in $L^{1}(\TT^{n})$, then $\phi\in L^{2}(\RR^{n})$, where $\phi$ is the refinable function associated with $\tau$. Thus, if it happens that one has additional information about the vanishing moment recovery function $S$, then it is not necessary to verify that $[\hat{\phi},\hat{\phi}]\in L^{\infty}(\TT^{n})$, even when $S$ is not a rational trigonometric polynomial. \hfill$\square$
\end{rem}

\section{Scaling Oblique Laplacian Pyramids}
\label{S:slp}

We now consider another perspective on Theorem~\ref{th:osqmf}, which is similar to the one described in the introduction for the UEP-based construction. This approach follows the one taken in \cite{SLP}, but rather than working with the Laplacian pyramid (LP) matrix coming from the polyphase components as was done there, in keeping with the current work, we consider the LP matrix with $\grp$-vector columns\footnote{To translate between these, consider the Fourier transform matrix $X(\omega)=[\ddet^{-1/2}e^{i(\omega+\gamma)\cdot\nu}]_{\gamma\in\Gamma^{*},\nu\in\Gamma}$. For a mask $g(\omega),$ if $G(\omega)=[g(\omega+\gamma)]_{\gamma\in\Gamma^{*}},$ then $[g_{\nu}(\dmat^{T}\omega)]_{\nu\in\Gamma}=X(\omega)^{*}G(\omega)$ (c.f. Equation~(\ref{eq:poly})).}. Given a lowpass mask $\tau$, let $H(\omega)=[\tau(\omega+\gamma)]_{\gamma\in\Gamma^{*}}$. Then we define the Laplacian pyramid matrix $\Phi_{\tau}(\omega)=[H(\omega)\;(I-H(\omega)H(\omega)^{*})X(\omega)],$ which is a $\ddet\times(\ddet+1)$ matrix with trigonometric polynomial entries and $\grp$-vector columns, and $X(\omega)=[\ddet^{-1/2}e^{i(\omega+\gamma)\cdot\nu}]_{\gamma\in\Gamma^{*},\nu\in\Gamma}$ is the Fourier transform matrix. When $\Phi_{\tau}(\omega)\Phi_{\tau}(\omega)^{*}=I,$ then by inspecting the entries of this matrix product, we see that $\tau$ satisfies the UEP conditions with the highpass masks $q_{\nu}(\omega)=\ddet^{-1/2}e^{i\omega\cdot\nu}-\tau(\omega)\overline{\tau_{\nu}(\dmat^{T}\omega)}.$ However, this requires $\tau$ to satisfy the restrictive QMF condition. Even when this condition does not hold, however, we can see that $$\Phi_{\tau}(\omega)\left[\begin{array}{c}H(\omega)^{*}\\ X(\omega)^{*}\end{array}\right]=[H(\omega)\;(I-H(\omega)H(\omega)^{*})X(\omega)]\left[\begin{array}{c}H(\omega)^{*}\\ X(\omega)^{*}\end{array}\right]=I,$$ so $\Phi_{\tau}$ has a right-inverse. But this right-inverse does not have the structure of a wavelet filter bank: the first row is a $\grp$-vector for the lowpass mask $\tau,$ and the remaining rows are $\grp$-vectors for some trigonometric polynomials $\tilde{q}_{\nu},\nu\in\Gamma$. Then $\tilde{q}_{\nu}(\omega)=\ddet^{-1/2}e^{i\omega\cdot\nu},$ so $\tilde{q}_{\nu}(0)=\ddet^{-1/2}\neq0$, and therefore $\tilde{q}_{\nu}$ are not wavelet masks. To correct this, we try scaling the matrix $\Phi_{\tau}$ to get a new filter bank satisfying the UEP conditions. We want a matrix $D(\omega)$ with trigonometric polynomial entries such that $\Phi_{\tau}(\omega)D(\omega)\Phi_{\tau}(\omega)^{*}=I.$ If $D(\omega)$ has a factorization as $B(\dmat^{T}\omega)B(\dmat^{T}\omega)^{*}$ for some $(\ddet+1)\times(r+1)$ matrix $B(\omega)$ with trigonometric polynomial entries, the product $\Phi'(\omega)=\Phi_{\tau}(\omega)B(\dmat^{T}\omega)$ will have $\grp$-vector columns and satisfy $\Phi'(\omega)\Phi'(\omega)^{*}=I$. Provided $\Phi'$ still has a lowpass mask generating its first column and highpass masks for the remaining columns, we will have a new collection of masks satisfying the UEP conditions, and an associated tight wavelet frame generated by these. In the introduction, we discussed some of the possible factorizations for $D(\omega)$ and the different constructions to which these lead. Now, we translate these ideas to the case of the OEP.

Suppose that $S$ is a nonnegative rational trigonometric polynomial satisfying $S(0)=1,$ and let $\tau$ be a lowpass mask satisfying the oblique sub-QMF condition with $S$. Recalling that $\Sigma(\omega)=\mathrm{diag}(S(\omega+\gamma))_{\gamma\in\Gamma^{*}}$, we define the $\ddet\times(\ddet+1)$ oblique Laplacian pyramid (OLP) matrix $\Phi_{S,\tau}(\omega)=[H(\omega)\;(\Sigma(\omega)-S(\dmat^{T}\omega)H(\omega)H(\omega)^{*})X(\omega)],$ where $X(\omega)$ is the Fourier transform matrix as above. Then $$\Phi_{S,\tau}(\omega)\left[\begin{array}{c} S(\dmat^{T}\omega)H(\omega)^{*}\\ X(\omega)^{*}\end{array}\right]= \Sigma(\omega),$$ so inverting $\Sigma(\omega),$ the matrix $\Phi_{S,\tau}$ has a right-inverse almost everywhere, in particular, wherever $S(\omega+\gamma)\neq0$ for all $\gamma\in\Gamma^{*}$. However, the second matrix in this product is once again not a wavelet filter bank, so as before, we will attempt to correct this by scaling the masks in $\Phi_{S,\tau}$ in order to get a new collection of masks satisfying the OEP conditions.

Let $a(\dmat^{T}\omega)=S(\dmat^{T}\omega)(2-S(\dmat^{T}\omega)H(\omega)^{*}\Sigma(\omega)^{-1}H(\omega)).$ Then \begin{align*}\Phi_{S,\tau}&(\omega)\begin{bmatrix}a(\dmat^{T}\omega)&0\\0&X(\omega)^{*}\Sigma(\omega)^{-1}X(\omega)\end{bmatrix}\Phi_{S,\tau}(\omega)^{*}\\&=a(\dmat^{T}\omega)H(\omega)H(\omega)^{*}+[\Sigma(\omega)-S(\dmat^{T}\omega)H(\omega)H(\omega)^{*}]\Sigma(\omega)^{-1}[\Sigma(\omega)-S(\dmat^{T}\omega)H(\omega)H(\omega)^{*}]\\&=a(\dmat^{T}\omega)H(\omega)H(\omega)^{*}+\Sigma(\omega)-2S(\dmat^{T}\omega)H(\omega)H(\omega)^{*}+S(\dmat^{T}\omega)^{2}(H(\omega)^{*}\Sigma(\omega)^{-1}H(\omega))H(\omega)H(\omega)^{*}\\&=\Sigma(\omega)+[a(\dmat^{T}\omega)+S(\dmat^{T}\omega)(-2+S(\dmat^{T}\omega)(H(\omega)^{*}\Sigma(\omega)^{-1}H(\omega)))]H(\omega)H(\omega)^{*}\\&=\Sigma(\omega).\end{align*} Note that $\Sigma(\omega)^{-1}$ is always well-defined as a rational trigonometric polynomial matrix so long as $S\not\equiv0,$ but if $S$ has zeroes, then $\Sigma(\omega)^{-1}$ will have poles. The assumptions of Theorem~\ref{th:osqmftwf} on $S$ also ensure that $\Sigma(\omega)^{-1}$ has no poles.

Then to obtain masks satisfying the OEP conditions, we want a factorization of the scaling matrix as \begin{equation}\label{eq:slpfact}\begin{bmatrix}a(\dmat^{T}\omega)&0\\0&X(\omega)^{*}\Sigma(\omega)^{-1}X(\omega)\end{bmatrix}=B(\dmat^{T}\omega) \begin{bmatrix}S(\dmat^{T}\omega)&0\\0&I\end{bmatrix} B(\dmat^{T}\omega)^{*},\end{equation} in which case $\Phi'(\omega)=\Phi_{S,\tau}(\omega)B(\dmat^{T}\omega)$ will have $\grp$-vector columns, and provided its first column is generated by a lowpass mask and the rest are highpass, we will have a new collection of masks satisfying the OEP conditions, since $$\Phi'(\omega)\begin{bmatrix}S(\dmat^{T}\omega)&0\\0&I\end{bmatrix}\Phi'(\omega)^{*}=\Phi_{S,\tau}(\omega)\begin{bmatrix}a(\dmat^{T}\omega)&0\\0&X(\omega)^{*}\Sigma(\omega)^{-1}X(\omega)\end{bmatrix}\Phi_{S,\tau}(\omega)^{*}=\Sigma(\omega).$$

Now we see that different factorizations of the scaling matrix lead to different constructions. In the proof of Theorem~\ref{th:osqmf}, we do not modify the lowpass mask, and add highpass masks corresponding to the sors generators for $a(\dmat^{T}\omega)-S(\dmat^{T}\omega)=S(\dmat^{T}\omega)^{2}(1/S(\dmat^{T}\omega)-H(\omega)^{*}\Sigma(\omega)^{-1}H(\omega))$. Then we might write the factorization of Equation~(\ref{eq:slpfact}) with $$B(\dmat^{T}\omega)=\begin{bmatrix}1&S(\dmat^{T}\omega)G(\dmat^{T}\omega)^{*}&0\\0&0&X(\omega)^{*}A(\omega)\end{bmatrix},$$ and $G(\omega),$ $A(\omega)$ are as in the proof of the theorem. Since the columns of $A(\omega)$ are $\grp$-vectors, $X(\omega)^{*}A(\omega)=[(a_{m})_{\nu}(\dmat^{T}\omega)]_{\nu\in\Gamma,1\leq m\leq M},$ where $a_{m}$ is the rational trigonometric polynomial generating the $m$th column of $A(\omega)$.

If instead there is a square root for $a(\omega)/S(\omega),$ so that $a(\omega)/S(\omega)=|g(\omega)|^{2}$ for all $\omega\in\TT^{n}$, then we might write the factorization of Equation~(\ref{eq:slpfact}) with $$B(\dmat^{T}\omega)=\begin{bmatrix}g(\dmat^{T}\omega)&0\\0&X(\omega)^{*}A(\omega)\end{bmatrix}.$$ This corresponds to modifying the lowpass mask to obtain $g(\dmat^{T}\cdot)\tau$.

A third possibility combines these two ideas, requiring a representation for $a(\omega)$ as $S(\omega)|g_{0}(\omega)|^{2}+S(\omega)^{2}\sum_{j=1}^{J}|g_{j}(\omega)|^{2}$, where $g_{0}(0)=1$. Then we might write the factorization of Equation~(\ref{eq:slpfact}) with $$B(\dmat^{T}\omega)=\begin{bmatrix} g_0(\dmat^{T}\omega)&S(\dmat^{T}\omega)G(\dmat^{T}\omega)^{*}&0\\0&0&X(\omega)^{*}A(\omega)\end{bmatrix},$$ obtaining a modified lowpass mask $g_{0}(\dmat^{T}\cdot)\tau$, as well as new highpass masks corresponding to the generators $g_{1},\ldots,g_{J}$.

As such, depending on the kinds of sors representations available for $a(\omega)$, and the criteria of the tight wavelet filter bank designer (such as whether or not the lowpass mask should be modified), some of these constructions may be preferable to others. Further investigation of possible factorizations of this scaling matrix may also lead to entirely new constructions.

\section{Examples}
\label{S:boxspl}

In this section, we consider the case of lowpass masks associated with box spline refinable functions. In particular, given a matrix $\Xi\in M_{n,d'}(\ZZ)$, the associated box spline refinable function is defined by its Fourier transform as $$\hat{\phi}_{\Xi}(\omega)=\prod_{j=1}^{d}\left(\frac{1-e^{-i\omega\cdot\xi_{j}}}{i\omega\cdot\xi_{j}}\right)^{m_{j}},$$ where $\xi_{j}$ is repeated $m_{j}$ times as a column of $\Xi$ (so for all $j$, $m_{j}\in\ZZ$ and $m_{j}>0$), and $\sum_{j=1}^{d}m_{j}=d'$. The associated lowpass mask is then given by $\tau=\tau_{\Xi}$, defined by $$\tau_{\Xi}(\omega)=\prod_{j=1}^{d}\left(\frac{1+e^{-i\omega\cdot\xi_{j}}}{2}\right)^{m_{j}}.$$

In the following examples, we show how the flexibility of our construction and the oblique sub-QMF condition can be used to trade off between some of the various desiderata in a tight wavelet frame. First, we describe the existing approach to a certain box spline, which gives 40 wavelet masks with maximum vanishing moments.

\begin{exmp}
\label{ex:manymasks}
We consider the mask from Example 6.4 in \cite{LaiStock}, which has direction matrix $\Xi=\begin{bmatrix}1&0&1&0&1\\0&1&0&1&1\end{bmatrix}=[I\,I\,e],$ where $e$ denote the vector of all ones. This corresponds to the lowpass mask $$\tau(\omega)=\left(\frac{1+e^{-i\omega_{1}}}{2}\right)^{2}\left(\frac{1+e^{-i\omega_{2}}}{2}\right)^{2}\left(\frac{1+e^{-i(\omega_{1}+\omega_{2})}}{2}\right).$$ In \cite{LaiStock}, the choice of $S$ used is $1/[\hat{\phi},\hat{\phi}]$, which satisfies the oblique QMF condition with $\tau$, and hence their construction gives a number of wavelet masks equal to $4\times K$, where $K$ is the number of sors generators for $1/S$. In Example 6.4 of \cite{LaiStock}, they find an sors representation for this $S$ with $K=10$, giving 40 wavelet filters with the maximum possible vanishing moments, which is 3. This construction is summarized in the first column of Table~\ref{fig:compare}.

In the case that the oblique QMF condition holds, the construction of Theorem~\ref{th:osqmf} agrees with the one in \cite{LaiStock}. Then if we use the same function $S$, applying Corollary~\ref{c:trigsors}, we see that the number of highpass masks can be improved to $16$, since we have the theoretical bound $K\leq 2^{n}$ (where here $n=2$). However, these sum of squares generators may be properly rational trigonometric polynomials (i.e., with nontrivial denominator), rather than trigonometric polynomials. This construction is summarized in the second column of Table~\ref{fig:compare}.\hfill$\square$
\end{exmp}

\begin{table}
\begin{tabular}{l|c|c|c|c|c|}
&Example 6.4 in \cite{LaiStock} &With Corollary~\ref{c:trigsors}&Extension to dim. $n$&Example~\ref{ex:fewmasks}&Example~\ref{ex:fewmasksextend}\\\hline
\# of hp masks&40&16&$4^{n}$ (upper bound)&8&$2^{n+1}$\\\hline
\# of vms&3&3&3&2&2\\\hline
\end{tabular}
\caption{Number of highpass masks and vanishing moments for the box spline in two dimensions with direction set $\Xi=[I\,I\,e]$, considered in Examples~\ref{ex:manymasks}, \ref{ex:fewmasks}, and generalized to $n$ dimensions in Example~\ref{ex:fewmasksextend}. The construction referred to in the first column is reviewed in Example~\ref{ex:manymasks}, where we also discuss the improvement provided by applying Corollary~\ref{c:trigsors}. The extension to $n$ dimensions referred to in the third column is described in Example~\ref{ex:fewmasksextend}.}
\label{fig:compare}
\end{table}

Now we present a method for dramatically reducing the number of frame generators by choosing a vanishing moment recovery function for which $K=1$, while only reducing the number of vanishing moments in the wavelet system from 3 to 2.

\begin{exmp}
\label{ex:fewmasks}
We continue to consider the box spline with direction set $\Xi=[I\,I\,e]$, in the case that $n=2$. Now we use another $S$, given by $$S(\omega)=\left[\frac{1}{54}(2+\cos(\omega_{1}))(2+\cos(\omega_{2}))(5+\cos(\omega_{1}+\omega_{2}))\right]^{-1}.$$ Since $1/S$ is a product of nonnegative univariate trigonometric polynomials, using the Fej\'{e}r-Reisz Lemma, we find that $$\frac{1}{S(\omega)}=\left|\frac{1}{24\sqrt{6}}(1+\sqrt{3}+(-1+\sqrt{3})e^{i\omega_{1}})(1+\sqrt{3}+(-1+\sqrt{3})e^{i\omega_{2}})(2+\sqrt{6}+(-2+\sqrt{6})e^{i(\omega_{1}+\omega_{2})})\right|^{2},$$ so $K=1$. Computing $f(S,\tau;\omega)$, we obtain $$\sum_{\gamma\in\{0,\pi\}^{2}}\cos^{4}\left(\frac{\omega_{1}+\gamma_{1}}{2}\right)\left(\frac{2+\cos(\omega_{1}+\gamma_{1})}{3}\right)\cos^{4}\left(\frac{\omega_{2}+\gamma_{2}}{2}\right)\left(\frac{2+\cos(\omega_{2}+\gamma_{2})}{3}\right)\sin^{4}\left(\frac{\omega_{1}+\omega_{2}+\gamma_{1}+\gamma_{2}}{2}\right).$$ When $\gamma=0$, we see that $\sin^{4}((\omega_{1}+\omega_{2})/2)$ has $4$ vanishing moments, and when $\gamma\neq0$, some $\gamma_{j}=\pi$, in which case $\cos^{4}((\omega_{j}+\gamma_{j})/2)=\sin^{4}(\omega_{j}/2)$ has $4$ vanishing moments, so $f(S,\tau;\cdot)$ has 4 vanishing moments. Letting $$g(\omega)=\cos^{2}\left(\frac{\omega_{1}}{2}\right)\left(\frac{1+\sqrt{3}+(-1+\sqrt{3})e^{i\omega_{j}}}{2\sqrt{3}}\right)\cos^{2}\left(\frac{\omega_{2}}{2}\right)\left(\frac{1+\sqrt{3}+(-1+\sqrt{3})e^{i\omega_{2}}}{2\sqrt{3}}\right)\sin^{2}\left(\frac{\omega_{1}+\omega_{2}}{2}\right),$$ we see that $$f(S,\tau;\omega)=\sum_{\gamma\in\{0,\pi\}^{2}}|g(\omega+\gamma)|^{2}=\sum_{\nu\in\{0,1\}^{2}}|(g_{\nu})(2\omega)|^{2},$$ where $g_{\nu}$ are the polyphase components of $g$. This latter representation gives the desired $\grp$-invariant sos representation in Equation~(\ref{eq:oepfsors}) with $4$ generators. Then our construction gives 8 wavelet masks generating a tight wavelet frame for $L^{2}(\RR^{2})$, each of which have at least 2 vanishing moments by Proposition~\ref{prop:oephpvms}. This example is summarized in the fourth column of Table~\ref{fig:compare}. \hfill$\square$
\end{exmp}

The method of the previous example can be extended to any box spline, in any number of dimensions. This approach gives a simple $S$, such that $1/S$ has an sos representation with one generator, that still gives a lower bound on the number of vanishing moments for all of the wavelet masks which is at least the accuracy number of any separable factor of the lowpass mask. In particular, when $\Xi$ is of the form $[I\, I\,\cdots\, I\, e]$, where the identity matrix $I$ appears $k$ times, and $e$ is the vector of all ones, then this approach gives a lower bound of $k$ vanishing moments, out of the maximum possible $k+1$. Since $1/S$ has a square root, this construction gives at most $2^{n+1}$ highpass masks, whereas the upper bound on the number of highpass masks for $S=1/[\hat{\phi},\hat{\phi}]$ is $4^{n}$, though in specific cases, it may be possible to find fewer highpass masks which still satisfy the OEP conditions for this $S$ and $\tau$. In fact, the ideas in this section can be extended to the even more general case of box splines with prime dilation factor, but to simplify the exposition, we have restricted our attention to dyadic dilation after the next example.

\begin{exmp}
To give an idea of the procedure in the setting of prime dilation, we briefly consider the piecewise-quadratic box spline with dilation factor 3 in 2 dimensions. In this case, we have the lowpass mask $$\tau(\omega)=\left(\frac{1+2\cos(\omega_{1})}{3}\right)^{2}\left(\frac{1+2\cos(\omega_{2})}{3}\right)^{2}\left(\frac{1+2\cos(\omega_{1}+\omega_{2})}{3}\right),$$ which has accuracy number 3 and flatness number 1. If we choose $S(\omega)$ as in Example~\ref{ex:fewmasks}, then we get $f(S,\tau;\omega)$ equal to \begin{align*}\sum_{\gamma\in(2\pi/3)\{0,1,2\}^{2}}&\left(\frac{1+2\cos(\omega_{1}+\gamma_{1})}{3}\right)^{4}\left(\frac{2+\cos(\omega_{1}+\gamma_{1})}{3}\right)\left(\frac{1+2\cos(\omega_{2}+\gamma_{2})}{3}\right)^{4}\left(\frac{2+\cos(\omega_{2}+\gamma_{2})}{3}\right)\\&\times\frac{16}{27}\sin^{4}\left(\frac{\omega_{1}+\omega_{2}+\gamma_{1}+\gamma_{2}}{2}\right)(5+4\cos(\omega_{1}+\omega_{2}+\gamma_{1}+\gamma_{2})).\end{align*} Arguing as in Example~\ref{ex:fewmasks}, we can show that this has 4 vanishing moments, and has a $\grp$-invariant sos representation with 9 sos generators, but in this case, we would be better off applying Corollary~\ref{c:trigsors} to $f(S,\tau;\omega/3)$, which is a nonnegative trigonometric polynomial, to obtain an sors representation with 4 generators, which then give a $\grp$-invariant sors representation for $f(S,\tau;\omega)$ with 4 generators. This means that our construction from Theorem~\ref{th:osqmf} gives 13 highpass masks, all of which have at least 2 vanishing moments. In this case, if we were to use $S=1/[\hat{\phi},\hat{\phi}]$, the upper bound on the number of highpass masks would be 36, and it is unclear whether this number could be reduced in this particular case. However, all of these highpass masks will have the maximum 3 vanishing moments.

In the setting of an $n$-dimensional box spline refinable function with dilation factor $p$ prime, the upper bound on the number of highpass masks when using $S=1/[\hat{\phi},\hat{\phi}]$ is $(2p)^{n}$, but using $S$ which is a product of univariate functions gives the upper bound of $2^{n}+p^{n}$.\hfill$\square$
\end{exmp}

\begin{exmp}
\label{ex:generalconst}

Now we present our general method in the case of dyadic dilation. We will argue by induction, with the base case being a direction matrix $\Xi$ containing a basis of integer vectors and their repeats, and the induction step being the introduction of a new integer vector (and possibly repeats of this vector). Together, this gives a procedure for constructing the vanishing moment recovery function $S$ and finding an sos representation for $f(S,\tau;\cdot)$.

\emph{Base Case:} Suppose $\Xi\in M_{n,d'}(\ZZ)$ is such that $\Xi=[\Xi_{0},\Xi']$, where $\Xi_{0}$ is invertible mod $2$, and the columns of $\Xi'$ are just repeats of $\{\xi_{1},\ldots,\xi_{n}\}$, the columns of $\Xi_{0}$, where $\xi_{k}$ is repeated $m_{k}$ times as a column of $\Xi$. Let $\Xi_{0}^{-1}\in M_{n}(\ZZ)$ such that $\Xi_{0}^{-1}\Xi_{0}\equiv I\pmod{2 M_{n}(\ZZ)}.$ Let $\ell\geq\mu=\min\{m_{k}:1\leq k\leq n\}$, and define $$S(\omega)=\left(\prod_{i=1}^{n}s_{1,m_{i},\ell}(\omega\cdot \xi_{i})\right)^{-1},$$ where
\begin{equation}
\begin{split}
s_{1,m_{k},\ell}(\omega)&>0\text{ for all }1\leq k\leq n,\label{eq:skconds1}\\
s_{1,m_{k},\ell}(2\omega)-\sum_{\gamma\in\{0,\pi\}}\cos^{2m_{k}}((\omega+\gamma)/2)s_{1,m_{k},\ell}(\omega+\gamma)&\geq0\text{ for }1\leq k\leq n,\text{ and }\\  
s_{1,m_{k},\ell}(2\omega)-\sum_{\gamma\in\{0,\pi\}}\cos^{2m_{k}}((\omega+\gamma)/2)s_{1,m_{k},\ell}(\omega+\gamma)&=O(|\omega|^{2\ell})\text{ for }\omega\approx0\text{ for }1\leq k\leq n.
\end{split}
\end{equation}

\begin{table}
\centering
\begin{tabular}{c|c|c|c|c|}
&$m=1$&$m=2$&$m=3$&$m=4$\\ \hline
$\ell=2$ & 1 & $2,1$&$33,26,1$&$29,28,3$\\ \hline
$\ell=3$ & 1 & $2,1$ &$33,26,1$&$29,28,3$\\ \hline
$\ell=4$ & 1 & $2,1$ &$33,26,1$&$1208,1191,120,1$\\ \hline
\end{tabular}
\caption{Coefficients for the trigonometric polynomials $s_{1,m,\ell}$ in Example~\ref{ex:generalconst}. For many pairs of $(m,\ell)$ here, the actual order of vanishing at $\omega=0$ exceeds the given $2\ell$. We list the coefficients without normalization, in the order $1,\cos(\omega),\cos(2\omega),\ldots$, so for example, $s_{1,2,2}(\omega)=(2+\cos(\omega))/3.$ The normalization is always chosen so that $s_{1,m,\ell}(0)=1$.}
\label{fig:trigpolysa1}
\end{table}

In Table~\ref{fig:trigpolysa1}, we give the coefficients for trigonometric polynomials satisfying all of these conditions for some values of $m_{k}$ and $\ell$. Now if we let $t_{k}=\sum_{\gamma_{k}\in\{0,\pi\}}\cos^{2m_{k}}((\omega\cdot\xi_{k}+\gamma_{k})/2)s_{1,m_{k},\ell}(\omega\cdot\xi_{k}+\gamma_{k})$ for all $1\leq k\leq n$, we see that \begin{align*}
\frac{1}{S(2\omega)}&-\sum_{\gamma\in\{0,\pi\}^{n}}\frac{|\tau(\omega+\gamma)|^{2}}{S(\omega+\gamma)}=\frac{1}{S(2\omega)}-\sum_{\gamma\in\{0,\pi\}^{n}}\frac{|\tau(\omega+\Xi_{0}^{-T}\gamma)|^{2}}{S(\omega+\Xi_{0}^{-T}\gamma)}\\
&=\prod_{k=1}^{n}s_{1,m_{k},\ell}(2\omega\cdot\xi_{k})-\sum_{\gamma\in\{0,\pi\}^{n}}\prod_{j=1}^{n}\cos^{2m_{j}}((\omega+\Xi_{0}^{-T}\gamma)\cdot\xi_{j}/2)s_{1,m_{j},\ell}((\omega+\Xi_{0}^{-T}\gamma)\cdot\xi_{j})\\
&=\prod_{k=1}^{n}s_{1,m_{k},\ell}(2\omega\cdot\xi_{k})-\prod_{j=1}^{n}t_{j}\\
&=\sum_{k=1}^{n}\left(\prod_{j=k+1}^{n}s_{1,m_{j},\ell}(2\omega\cdot\xi_{j})\right)\left(\prod_{\ell=1}^{k-1}t_{\ell}\right)(s_{1,m_{k},\ell}(2\omega\cdot\xi_{k})-t_{k}).
\end{align*}
Each term in this sum is a product of nonnegative, $\pi$-periodic, univariate factors, which by the Fej\'{e}r-Riesz Lemma means that $f(S,\tau;\cdot)$ has a sum of squares representation with $n$ $\grp$-invariant squares. Moreover, by the last property of Equation~(\ref{eq:skconds1}), looking at the last factor of the terms in this sum, $f(S,\tau;\cdot)$ is $O(|\omega|^{2\ell})$ for $\omega\approx0$.

\emph{Induction Step:} Now suppose that $\Xi=[\Xi_{1},\xi,\xi,\ldots,\xi]$, where $\xi$ is repeated $m$ times, and is not a column of $\Xi_{1}$. Then $\tau(\omega)=\tau_{1}(\omega)(2^{-1}(1+e^{-i\omega\cdot\xi}))^{m}$, and we let $S(\omega)=S_{1}(\omega)s_{2,m,\ell}(\omega\cdot\xi)^{-1}.$ By the induction hypothesis, $S_{1}$ is such that $f(S_{1},\tau_{1};\cdot)\geq0$ and $O(|\omega|^{2\mu})$ for $\omega\approx0$, and $S_{1}^{-1}$ has a square root. Recall from the base case that $\mu$ is the minimum of the multiplicities of the first $n$ columns of $\Xi_{1}$, which form a basis for $\ZZ^{n}$ mod $2\ZZ^{n}$. We choose $\ell\geq\mu$, and $s_{2,m,\ell}$ satisfying the following conditions:
\begin{equation}
\begin{split}
s_{2,m,\ell}(\omega)&>0,\label{eq:skconds2}\\
s_{2,m,\ell}(2\omega)-\cos^{2m}(\omega/2)s_{2,m,\ell}(\omega)&\geq0,\text{ and}\\
s_{2,m,\ell}(2\omega)-\cos^{2m}(\omega/2)s_{2,m,\ell}(\omega)&=O(|\omega|^{2\ell})\text{ for }\omega\approx0.\end{split}
\end{equation}

\begin{table}
\centering
\begin{tabular}{c|c|c|c|c|}
&$m=1$&$m=2$&$m=3$&$m=4$\\ \hline
$\ell=2$ & $5,1$ & $2,1$ &$33,26,1$&$29,28,3$\\ \hline
$\ell=3$ & $97,24,-1$ & $237,124,-1$&$33,26,1$&$29,28,3$ \\ \hline
$\ell=4$ & $24134,6513,-438,31$ & $4927,267,-51,5$ &$8306,6567,246,1$&$1208,1191,120,1$\\ \hline
\end{tabular}

\caption{Coefficients for the trigonometric polynomials $s_{2,m,\ell}$ in Example~\ref{ex:generalconst}. For many pairs of $(m,\ell)$ here, the actual order of vanishing at $\omega=0$ exceeds the given $2\ell$. We list the coefficients without normalization, in the order $1,\cos(\omega),\cos(2\omega),\ldots$, so for example, $s_{2,1,2}(\omega)=(5+\cos(\omega))/6$. The normalization is always chosen so that $s_{2,m,\ell}(0)=1$.}
\label{fig:trigpolysa2}
\end{table}

In Table~\ref{fig:trigpolysa2}, we give the coefficients for trigonometric polynomials satisfying all of these conditions for some values of $m$ and $\ell$. Then
\begin{align*}
f(S,\tau;\omega)&=\frac{s_{2,m,\ell}(2\omega\cdot\xi)}{S_{1}(2\omega)}-\sum_{\gamma\in\{0,\pi\}^{n}}\frac{|\tau_{1}(\omega+\gamma)|^{2}}{S_{1}(\omega+\gamma)}\cos^{2m}((\omega+\gamma)\cdot\xi/2)s_{2,m,\ell}((\omega+\gamma)\cdot\xi)\\
&=f(S_{1},\tau_{1};\omega)s_{2,m,\ell}(2\omega\cdot\xi)\\&\quad+\sum_{\gamma\in\{0,\pi\}^{n}}\frac{|\tau_{1}(\omega+\gamma)|^{2}}{S_{1}(\omega+\gamma)}\left(s_{2,m,\ell}(2\omega\cdot\xi)-\cos^{2m}((\omega+\gamma)\cdot\xi/2)s_{2,m,\ell}((\omega+\gamma)\cdot\xi)\right),
\end{align*}
so by the Fej\'{e}r-Riesz Lemma, this has a sum of squares representation with $\#\{\text{sos generators for }f(S_{1},\tau;\cdot)\}+2^{n}$ sos generators. From the last property in Equation~(\ref{eq:skconds2}), when $\gamma=0$, the corresponding term in the sum above is $O(|\omega|^{2\ell})\leq O(|\omega|^{2\mu})$. When $\gamma\neq0$, $|\tau_{1}(\omega+\gamma)|^{2}$ has factors $\cos^{2\mu}((\omega+\gamma)\cdot\xi_{i}/2)$ for $\{\xi_{1},\ldots,\xi_{n}\}$ making a basis for $\ZZ^{n}$ mod $2\ZZ^{n}$, so some $\gamma\cdot\xi_{i}\equiv \pi\pmod{2\pi}$, which gives $\cos^{2\mu}((\omega+\gamma)\cdot\xi_{i}/2)=\sin^{2\mu}(\omega\cdot\xi_{i}/2)$, and this term has at least $2\mu$ vanishing moments. Since $f(S_{1},\tau_{1};\omega)=O(|\omega|^{2\mu})$ by the induction hypothesis, we see that $f(S,\tau;\cdot)$ has at least $2\mu$ vanishing moments. Then the construction of Theorem~\ref{th:osqmf} gives $\#\{\text{sos generators for }f(S_{1},\tau_{1};\cdot)\}+2^{n+1}$ highpass masks (noting that $1/S=|g|^{2}$ for a rational trigonometric polynomial $g$), which by Proposition~\ref{prop:oephpvms} have at least $\mu$ vanishing moments, since the accuracy number of $\tau$ is at least $\mu$.

By induction, we see that if $\Xi$ has $d$ distinct columns, where the first $n$ form a basis for $\ZZ^{n}$ mod 2, and the minimum of the multiplicities of these first $n$ columns is $\mu$, then we can find $n+2^{n}(d-n+1)$ highpass masks using the method of Theorem~\ref{th:osqmf}, and these will all have at least $\mu$ vanishing moments. \hfill$\square$
\end{exmp}

In \cite{AlgPersII}, a construction based on the UEP obtained the same number of highpass masks generating a tight wavelet frame with any box spline refinable function, but their construction always has some masks having only one vanishing moment, whereas ours typically gives more. We now extend Example~\ref{ex:fewmasks} to the case of arbitrary dimension.

\begin{exmp}
\label{ex:fewmasksextend}
In Example~\ref{ex:fewmasks}, we found $s_{1,2,2}(\omega)=\frac{1}{3}(2+\cos(\omega))$ satisfying $s_{1,2,2}(2\omega)-\cos^{4}(\omega/2)s_{1,2,2}(\omega)-\sin^{4}(\omega/2)s_{1,2,2}(\omega+\pi)=0$, which is why we only have $2^{2}(3-2+1)=8$ highpass masks all having at least $\mu=2$ vanishing moments. In $n$ dimensions, for the box spline with direction matrix $[I\, I\, e]$, where $e$ is the vector of all ones, the natural extension of $S$ from Example~\ref{ex:fewmasks} is $$S(\omega)=\left[\prod_{k=1}^{n}\left(\frac{2+\cos(\omega_{k})}{3}\right)\left(\frac{5+\cos(\omega\cdot e)}{6}\right)\right]^{-1},$$ and this gives us $2^{n+1}$ highpass masks all having 2 vanishing moments. If we were to use $1/[\hat{\phi},\hat{\phi}]$ instead, it is unclear whether it is possible to use fewer highpass masks than the upper bound $4^{n}$, though these will all have the maximum 3 vanishing moments. This example is summarized in the third and last columns of Table~\ref{fig:compare}. \hfill$\square$
\end{exmp}

%\bibliographystyle{plain}
%\bibliography{OEPsubQMFBibliography}{}

\end{document}